%% file: AMCF_isoparametric_v9.tex
 \documentclass[11pt, reqno]{amsart}
\usepackage{amsmath,amsthm,amssymb, graphicx}
\input header_nosubsection.tex

\newcommand{\balign}{\begin{align}}
\newcommand{\ealign}{\end{align}}
\newcommand{\baligns}{\begin{align*}}
\newcommand{\ealigns}{\end{align*}}

\def\min{{\rm min\/}}
\def\bc{{\bf c}}

\newcommand{\vn}{{\bf n}}
\newcommand{\Id}{{\rm Id\/}}

\begin{document}

\title[Mean curvature flow]
{Ancient solutions to mean curvature flow for isoparametric submanifolds}
\author{Xiaobo Liu$^\ast$}\thanks{$^\ast$Research was partially supported by NSFC grants 11431001 and 11890662}
\address{Beijing International Center for Mathematical Researcch  \& School of Mathematical Sciences
\\ Peking University, Beijing, China}
\email{xbliu@math.pku.edu.cn}
\author{Chuu-Lian Terng}
\address{Department of Mathematics\\
University of California at Irvine, Irvine, CA 92697-3875}
\email{cterng@math.uci.edu}

\begin{abstract}

Mean curvature flow for isoparametric submanifolds in Euclidean spaces and spheres was studied in \cite{LT}.
In this paper, we will show that all these solutions are ancient solutions. We also discuss rigidity of ancient mean curvature flows for  hypersurfaces in spheres and its relation to the Chern's conjecture on the norm of the second fundamental forms of minimal hypersurfaces in spheres.
\end{abstract}

\maketitle

\section{ Introduction}
\ms

The mean curvature flow (abbreviated as {\it MCF}) of a submanifold $M$ in a Riemannian manifold $X$
over an interval $I$ is a map
$ f: I \times M \longrightarrow X $ satisfying
$$\frac{\p f}{\p t} = H(t, \cdot),$$
where $H(t, \cdot)$ is the mean curvature vector field  of $f(t, \cdot)$.
If a solution to this equation exists for all $t \in (-\infty, T)$ for some $T \geq 0$, then it is called an
{\it ancient solution}. Ancient solutions are important in studying singularities of MCF.
A simple example of ancient solution to MCF is the shrinking sphere in a Euclidean space.
A set of conditions which ensure a compact ancient solution to be the shrinking sphere is given in \cite{HS}.
Other examples of compact convex ancient solutions for MCF of nonconvex hypersurfaces in Euclidean spaces can be found in
\cite{A}, \cite{BLT}, \cite{HH}, \cite{W}, \cite{Wh}, etc.
Recently an ancient solution of MCF of hypersurfaces with the topology of $S^1 \times S^{n-1}$ in $\mathbb{R}^{n+1}$
was given in \cite{BLM}.
A construction of higher codimensional curve
shortening flows was given in \cite{AAAW} and \cite{SS}. It was proved in \cite{QT} that after reparametrization the family of proper Dupin submanifolds in sphere constructed in \cite{PT} is a MCF for submanifolds in spheres.

In this paper, we will give a class of
 ancient solutions to MCF in Euclidean spaces and spheres for compact submanifolds. These examples include both hypersurfaces and higher
 codimensional submanifolds in spheres and have more complicated topological types.

A submanifold $M$ of a space form is {\it isoparametric} if
its normal bundle is flat and principal curvatures along any parallel normal
vector field are constant. The following results were proved in \cite{T}:
\ben
\item[(i)] If $M$ is a compact isoparametric submanifold in $\R^{n+k}$, then $M$ is contained in a hypersphere. \item[(ii)] The set of parallel submanifolds to $M$ forms a singular foliation,
whose top dimensional leaves are also isoparametric and lower dimensional leaves are smooth focal submanifolds of $M$. Focal submanifolds are no longer isoparametric.
\een

A submanifold $M$ in $\R^{n+k}$ is {\it full} if $M$ is not contained in any hyperplane.
The {\it rank} of a full isoparametric submanifold in a Euclidean space is the co-dimension of the submanifold. Compact rank 2 isoparametric submanifolds in Euclidean spaces are isoparametric hypersurfaces in spheres.  These hypersurfaces have rich topology.  For example, principal orbits of isotropy representations of rank 2 symmetric spaces are homogeneous examples of isoparametric hypersurfaces in spheres. For each orthogonal representation of Clifford algebra, Ferus-Karcher-M\"unzner constructed in \cite{FKM} a family of isoparametric hypersurfaces in spheres. Most of these examples are not homogeneous.  Principal orbits of isotropy representations of higher rank  symmetric spaces are isoparametric submanifolds of higher codimension.

In \cite{LT}, we studied MCF with initial data an isoparametric submanifold in both Euclidean spaces and in spheres.
 We call the MCF in spheres (respectively Euclidean spaces) the {\it spherical MCF\/} (respectively {\it Euclidean MCF\/}).
 Let $f(t,x)$ and $F(t,x)$ denote the spherical and Euclidean MCF with initial data the inclusion map
 $ f_0: M \to S^{n+k-1}  $
 of an $n$-dimensional isoparametric submanifold
 $M$ in the unit sphere $S^{n+k-1}\subset \mathbb{R}^{n+k}$.
 The following results were proved in \cite{LT}:
\ben
\item  $f(t, \cdot)$ and $F(t,\cdot)$ are isoparametric and parallel to $M$.
\item If $M$ is not minimal in $S^{n+k-1}$, then the spherical MCF collapses in finite time $T > 0$ to a lower dimensional focal submanifold $N \subset S^{n+k-1}$ and the Euclidean MCF is
\begin{equation} \label{eqn:aa}
F(t,x)= \sqrt{1-2nt}\, f(-\frac{1}{2n} \ln (1-2nt), x).
\end{equation}
In particular, the Euclidean MCF collapses at $T_0= \frac{1-e^{-2nT}}{2n}$ to the focal submanifold $e^{-nT} N$.
Moreover  $T_0 <\frac{1}{2n}$.
\item If $M$ is a minimal isoparametric submanifold of $S^{n+k-1}$, then the spherical MCF $f(t,x)= f_0(x)$ is stationary,  and the Euclidean MCF is
\begin{equation} \label{eqn:aa3}
F(t,x)= \sqrt{1-2nt}\, f_0(x),
\end{equation}
 which homothetically collapses to a point at $T_0= \frac{1}{2n}$.
\een

One of the main results of this paper is to show that the above spherical and Euclidean MCFs are ancient solutions:

\begin{thm}\label{bo} Let $M$ be an isoparametric submanifold in the unit sphere, $f(t,x)$, $F(t, x)$ the spherical and Euclidean
MCF with initial data $M$. Then we have the following:
\ben
\item $f(t,\cdot)$ and $F(t, \cdot)$ exist for all $t\in (-\infty, 0]$.
\item There is a unique minimal isoparametric submanifold $M_\min$ in $S^{n+k-1}$, which is parallel to $M$. In fact, there exists a unit parallel normal vector field $\zeta$ on $M$ in $S^{n+k-1}$ such that  the map $h:M\to S^{n+k-1}$ defined by
$$h(x)= (\cos r) x + (\sin r) \zeta(x)$$
is the embedding of  $M_\min$ in $S^{n+k-1}$, where $r$ is the spherical distance between $M$ and $M_\min$.
\item
\begin{align}
&\lim_{t\to -\infty} ||F(t,x)- \sqrt{1-2nt} \, \, h(x)||=0, \label{ax1}\\
& \lim_{t\to -\infty} ||f(t,x)- h(x)||= 0. \label{ax2}
\end{align}
for all $x \in M$.
 \een
\end{thm}

 If $M$ is minimal in the sphere, $h$ is just the identity map. Part (3) of this theorem means that the MCF of $M$ converges to MCF of $M_{\rm min}$ in a suitable sense as $t \rightarrow -\infty$ . In particular, the spherical MCF of $M$
always converges to a minimal isoparametric submanifold as $t \rightarrow -\infty$.

Comparing to higher codimensional cases, MCF of hypersurfaces are expected to be more rigid and have attracted more attention in the literature. In this paper, we will give a simple unified explicit solution for MCF of all isoparametric hypersurfaces in the sphere (see Propositions \ref{lem:MCF2} and  \ref{cor:SMCF2}). We can use such solutions to obtain geometric descriptions of MCF for concrete
examples of isoparametric hypersurfaces (see examples \ref{ex:torus} and \ref{ex:g3n3} for the cases of Clifford tori and flag manifolds).

It is well known (cf. \cite{Mu}) that isoparametric hypersurfaecs in the sphere can have only $g$ distinct principal curvatures
with $g=1,2,3,4,6$. Isoparametric hypersurfaces with $g=1$ are precisely hyperspheres (i.e. codimensional one subspheres). If it is not totally geodesic, the spherical MCF of a hypersphere collapses to a point in finite positive time and tends to a totally geodesic hypersphere as $t\to -\infty$. Huisken and Sinestrari called these ancient flows the {\it shrinking spherical cap} in \cite{HS} and
 proved the following rigidity results (see Theorem 6.1 and Remark 6.2 in \cite{HS}):
 The spherical MCF $M_t$ for hypersurfaces in $S^{n+1}$ with non-minimal $M_0$ must be a shrinking spherical cap if one of the following  conditions is satisfied for
 all $t<0$:
 \ben
\item There exists $C>0$ such that
\beq\label{ah}
0< ||A(t)||^2< C||H(t)||^2.
\eeq
\item
For some constant $B<4n$,
\beq\label{ai}
||A(t)||^2 < e^{-Bt} ||H(t)||^2.
\eeq

\item
\beq\label{at}
||A(t)||^2 -\frac{1}{n-1} ||H(t)||^2 \leq 2.
\eeq

\een
In above conditions, $A(t)$ and $H(t)$ are the shape operator (or equivalently the second fundamental form) and mean curvature vector field for $M_t$ respectively.
Higher codimensional rigidity results modeled on shrinking spherical caps were obtained in \cite{LN},  \cite{RS}, and \cite{LXZ}.

Motivated by above rigidity results as well as results in \cite{AdC}, we give estimates for $\frac{||A(t)||^2}{||H(t)||^2}$ and $||A(t)||^2- \frac{1}{n} ||H(t)||^2$ in section \ref{sec:rank2} for spherical MCF of isoparametric hypersurfaces with $g$ disticnt principal curvatures. For $g=1$, Huisken-Sinestrari's theorem and results in \cite{LXZ} give evidence that these estimates give rigidity of ancient solutions of MCF. This leads us to conjecture that these estimates will give us the rigidity results modeled on spherical MCF for isoparametric hypersurfaces with $g\geq 2$ distinct principal curvatures (see Conjectures \ref{conj:A/H} and \ref{conj:ay}). The rigidity of the stationary case of the spherical MCF for hypersurfaces is related to Chern's conjecture on the norm square of the second fundamental form of minimal hypersurfaces in spheres.
Moreover, we will see that condition \eqref{ai} in Huisken-Sinestrari's theorem above is sharp in Remark \ref{prop:sharp2}.
We will also discuss the sharpness of condition \eqref{at} in Remark \ref{rem:sharp3}.

This paper is organized as follows:  We prove Theorem \ref{bo} in section \ref{sec:pfmain}, and compute the norms of  shape operators and mean curvature vector fields of isoparametric submanifolds in section \ref{sec:norm}. In the last section, we give explicit solutions of the spherical MCF for isoparametric hypersurfaces, compute formulas for the mean curvature vector and the norm of the shape operators, and discuss the rigidity question of these ancient flows and its relation to Chern's conjecture.

The authors would like to thank Carlo Sinestrari for helpful discussions.

\bs
\section{Proof of Theorem \ref{bo}}
\label{sec:pfmain}

Let $M$ be an $n$-dimensional compact full isoparametric submanifold in a Euclidean space $\mathbb{R}^{n+k}$. Without loss of generality,
we may assume $M^n$ is contained in the unit sphere $S^{n+k-1}$ centered at the origin. We write down some results in \cite{T} which
will be needed later:
\ben
\item[(i)] The tangent bundle of $M$ has the decomposition
$$TM=\oplus_{i=1}^g E_{i}$$
where $E_i$ are called {\it curvature distributions}.
The {\it curvature normals} of $M$  are parallel normal vector fields $\vn_{i}$
such that the shape operator $A_{\xi}$ of $M$ as a submanifold of Euclidean space is given by
\begin{equation} \label{eqn:curvnorm}
 A_{\xi} |_{E_{i}} = \li \xi, \vn_{i}\ri  \Id_{E_i}
 \end{equation}
for all normal vector $\xi$.
 The {\it multiplicity} of $\vn_{i}$ is defined to be the dimension of $E_i$ and is denoted by $m_{i}$.
 \item[(ii)] Each $E_i$ is an integrable distribution and its leaves are standard $m_i$ dimensional spheres.
\item[(iii)] Fix $x_0\in M$,  the group generated by hyperplanes $\li \xi, \vn_i\ri =0$ in $\nu_{x_0}M$ for $1\leq i\leq g$ is a Weyl group $W$.
Let $C$ be the open Weyl chamber of $W$ which contains $x_0$. Then $C$ is given by
\begin{equation} \label{eqn:chamber}
 C=\{ x \in \nu_{x_0}M \,\, \mid \,\, <x, \vn_i> \, < \, 0 \,\,\, for \,\,\, all \,\,\, i \}.
\end{equation}
Let $\overline{C}$ be the closure of
$C$ and $k$ the dimension of $\nu_{x_0}M$. $\overline{C}$ is a simplicial cone with vertex at the origin.
\item[(iv)]
Given $\xi\in \bar C$, there exists a unique parallel normal vector field $\ti \xi$ such that $\ti\xi(x_0)= \xi-x_0$ and the submanifold parallel to $M$ defined by $\ti \xi$ is
$$M_\xi= \{x+ \ti\xi(x)\n x\in M\}.$$
 Moreover, we have the following:
\ben
\item If $\xi\in C$, then $M_\xi$ is diffeomorphic to $M$ and is also isoparametric.
\item Let $\partial C$ be the boundary of $C$. If $x\in \partial C$,  then $M_\xi$ is a smooth lower dimensional focal submanifold of $M$.
\item $M_\xi\cap \nu_{x_0}M$ is the $W$-orbit at $\xi$, where $W$ is the Weyl group associated to $M$.
\item $\cup\{M_\xi\n \xi\in \bar C\}$ is a singular foliation of $\R^{n+k}$, called isoparametric foliation of $\R^{n+k}$.
\item $\cup\{M_\xi\n \xi\in \bar C, ||\xi||=1\}$ is a singular foliation of $S^{n+k-1}$, called isoparametric foliation of $S^{n+k-1}$.
\item Given a unit vector  $\xi\in C$ and $\xi\not=x_0$. There exists a unique unit vector $y_0\in \nu_{x_0}M$ such that $(\cos r) x_0+ (\sin r) y_0= \xi$, where $r$ is the distance between $x_0$ and $\xi$ in $S^{n+k-1}$. Let $\eta$ denote the unique parallel unit normal vector field on $M$ satisfying
    $\eta(x_0)= y_0$. Then
$$M_\xi= \{(\cos r) x+ (\sin r) \eta(x)\n x\in M\}.$$
\item If $\xi\in \p C$ and lies in only one reflection hyperplane $\li \xi, \vn_i\ri=0$, then $\dim(M_\xi)= n-m_i$ and the map $h_{\xi}:M\to M_\xi$ defined by
    \[ h_{\xi}(x):= (\cos r) x+ (\sin r) \eta(x)\]
    is a fibration with fibers isometric to a round $m_i$ dimensional sphere (fibers are leaves of $E_i$).
\een
\item[(v)] All curvature normals $\vn_1, \ldots, \vn_g$ at $x_0$ span $\nu_{x_0}M$.
\een

Fix $x_0 \in M$. It was shown in \cite[Theorem 2.2]{LT} that MCF of $M$ is through parallel isoparametric submanifolds and the Euclidean MCF with initial data $M$ becomes the following ODE for $x(t)\in \nu_{x_0}M$:
 \begin{equation} \label{eqn:EMCF}
  x'(t) = - \sum_{i-1}^g \frac{m_i \vn_i}{<x(t), \vn_i>}, \quad x(0)=x_0.
 \end{equation}

As in the proof of \cite[Theorem 2.4]{LT},
there is a map
\[ \begin{array}{cccl} P:  & \overline{C} & \longrightarrow & \mathbb{R}^k \\
                            & x & \longmapsto & (P_1(x), \ldots , P_k(x))
                            \end{array}
\]
where $P_i$ are $W$-invariant polynomials on $\nu_{x_0}M$ and the map $P$ is a
homeomorphism from $\overline{C}$ to $B:= P(C) \subset \mathbb{R}^k$.
Under this homeomorphism, the Euclidean MCF \eqref{eqn:EMCF} becomes a flow equation
along a polynomial vector field on $\mathbb{R}^k$. Solution to the latter equation always exists for all $t$.
In fact such solutions can be recursively constructed as in \cite[Section 3]{LT}. Consequently, the solution
to Euclidean MCF \eqref{eqn:EMCF} exists for $t$ as long as $x(t)$ does not hit $\partial C$.

Since dimensions of leaves through boundary points of $C$ are smaller than $n$, the $n$-dimensional volume of the leaves of the isoparametric foliation are $0$ on
$\partial C$.  Note that the MCF decreases volume. For $t \leq 0$, $x(t)$ will stay away from
$\partial C$.
Consequently, $x(t)$ exists for all $t \in (-\infty, 0]$. It was also proved in \cite{LT} that $x(t)$ exists for $t \in [0, T)$ for
a positive $T < \infty$ and $x(t)$ converges to a point in $\partial C$ as $t \rightarrow T$.
So the maximal time interval for the existence of solutions of MCF for $M$ is $(-\infty, T)$ and this solution is indeed an ancient solution. This proves part (1) of Theorem \ref{bo} for the Euclidean MCF.

Let $\tilde{M}^n \subset S^{n+k-1}$ be a minimal isoparametric submanifold  in the same isoparametric foliation as $M$, and
$\tilde{x}_0$ the unique point in the intersection of $\tilde{M}$ and $C$.
As mentioned in (iv) above that there exists a parallel unit normal vector field $\zeta$ to $M$ in $S^{n+k-1}$ such that
$$\ti M= M_{\ti x_0}= \{(\cos r)\, x+ (\sin r)\, \zeta(x)\n x\in M\}.$$
Here $r$ is the distance between $M$ and $M_{\ti x_0}$ in $S^{n+1}$.
The map
\[ h(x)=(\cos r)\, x+ (\sin r)\, \zeta(x) \]
defines a diffeomorphism $h: M \rightarrow \ti M$.
Later we will show that minimal isoparametric submanifold is unique in each isoparametric foliation of the sphere and hence completes the proof of part (2) of Theorem \ref{bo}.

Let $\tilde{x}(t) \in C$  be the solution to the Euclidean MCF of  $\tilde{M}$ with $\tilde{x}(0) = \tilde{x}_0$.
By equation \eqref{eqn:aa3}, $\tilde{x}(t)$ is given by
\beq\label{ac}
 \tilde{x}(t) = \sqrt{1-2nt} \,\,\, \tilde{x}_0.
 \eeq
In particular $\tilde{x}(t) $ is parallel to $\tilde{x}_0$ for all $t$ and
\[ \|\tilde{x}(t) \|= \sqrt{1-2nt} \,\, . \]

Let $x(t)$ be the solution of the Euclidean MCF \eqref{eqn:EMCF} with initial data $x_0$, and $\ti x(t)$ the solution given by \eqref{ac} (i.e., the Euclidean MCF with initial data a minimal isoparametric submanifold).
To prove \eqref{ax1}, we compute the derivative of
$$D(t):= ||x(t)-\ti x(t)||^2,$$
to get
 \begin{equation}\label{bp}
\frac{1}{2} \frac{d}{dt}\,  D(t)= \sum_{i=1}^g m_i\, \frac{\li x(t)-\tilde{x}(t), \vn_i\ri^2}{\li x(t), \vn_i\ri \li \tilde{x}(t), \vn_i\ri} \geq 0
\end{equation}
(cf. \cite[Equation 2.8]{LT}). Let
\[ \alpha(t) = \frac{x(t)-\tilde{x}(t)}{||x(t)-\tilde{x}(t)||}. \]
Then $\alpha(t)$ is a unit vector in $\nu_{x_0}M$ since both $x(t)$ and $\tilde{x}(t)$ lie in $C$.
By equation \eqref{eqn:aa},
$$\|x(t)\|=\sqrt{1-2nt}.$$
Hence
equation~\eqref{bp} implies
\begin{eqnarray*}
\frac{1}{2} \frac{d}{dt}\,  D(t)
& \geq & D(t) \,\, \sum_{i=1}^g m_i\, \frac{ \li \alpha(t), \vn_i \ri^2}{\| x(t) \| \cdot \|\tilde{x}(t)\| \cdot \|\vn_i\|^2 } \\
&=& \frac{D(t)}{1-2nt} \,\, \sum_{i=1}^g m_i\, \frac{ \li \alpha(t), \vn_i \ri^2}{ \|\vn_i\|^2 }.
\end{eqnarray*}
We claim that the set
\[ \left\{ \left. \sum_{i=1}^g m_i\, \frac{ \li \beta, \vn_i \ri^2}{ \|\vn_i\|^2 } \,\,\, \right| \,\,\,  \beta \in \nu_{x_0}M,
        \,\,\, ||\beta||=1 \right\} \]
is bounded from below by a positive constant. Otherwise, by the compactness of the sphere,
 there would exist a unit vector in $\nu_{x_0}M$ which is perpendicular to all curvature
normals $\vn_i$. This would contradict to the assumption that $M$ is full.

Hence there exists a constant $b>0$ such that
\begin{eqnarray*}
\frac{d}{dt}\,  D(t)
& \geq &  \frac{b}{1-2nt} D(t) .
\end{eqnarray*}
Integrating $\frac{1}{D(t)} \frac{d}{dt}\,  D(t)$ over an interval $[a, 0]$ with $a < 0$, we have
\[ D(a) \leq D(0) (1-2na)^{-\frac{b}{2n}} \]
for all $a < 0$. Hence $D(a) \rightarrow 0$ as $a \rightarrow - \infty$.
This shows that
\begin{equation} \label{eqn:limE}
 \lim_{t \rightarrow - \infty} (x(t)-\tilde{x}(t)) = 0.
\end{equation}
So the Euclidean MCF for $M$ and $\tilde{M}$ are asymptotically the same as $t$ goes to $ - \infty$. This proves equation \eqref{ax1}.

Let $y(t)$ be the solution to the spherical MCF for $M$ with $y(0)=x_0$. By equation \eqref{eqn:aa},  $x(t)$ and $y(t)$ are related by
\begin{equation} \label{eqn:ES}
 x(t) = \sqrt{1-2nt} \,\,\, y\left(-\frac{\ln(1-2nt)}{2n}\right).
\end{equation}
 Since $x(t)$ exists for all $t \in (-\infty, 0]$, so does $y(t)$.
 Hence the spherical MCF for $M$ is also an ancient solution.
 This proves part (1) of Theorem~\ref{bo} for spherical MCF.

Let $M_1$ and $M_2$ be two distinct isoparametric submanifolds  in the same isoparametric foliation in $S^{N-1}$,
 $y_1(t) \in C$ and $y_2(t) \in C$ solutions to spherical MCF of  $M_1$ and $M_2$ respectively.
Define $$ f(t):= ||y_1(t)-y_2(t)||^2. $$
By \cite[Equation (5.2)]{LT},
\[ f'(t) \geq 2n f(t) \]
for all $t$.
Hence \[ \left( \ln f(t) \right)' \geq 2n \]
for all $t$. Integrating both sides over an interval $[a, 0]$ with $a \leq 0$, we obtain
\[ \ln f(0) - \ln f(a) \geq -2na. \]
Therefore we have
\[ 0 \leq f(a) \leq f(0) e^{2na} \]
for all $a \leq 0$.
Consequently
\[ \lim_{a \, \rightarrow \, -\infty} f(a) = 0. \]

If we take $M_2$ to be a minimal isoparametric submanifold in $S^{N-1}$ (which exists by a result in \cite{T})
and $M_1$ an arbitrary isoparametric submanifold in $S^{N-1}$, then $y_2(t)$ is constant and the above arguments imply that
 \[ \lim_{t \, \rightarrow \, -\infty} y_1(t) = y_2(0). \]
 This shows that the spherical MCF of an arbitrary isoparametric submanifold in $S^{N-1}$ converges to a minimal
 isoparametric submanifold as $t \rightarrow -\infty$. This proves equation \eqref{ax2}.

 The above arguments also show that minimal isoparametric submanifold is unique in each isoparametric
family in a sphere. If fact, if both $M_1$ and $M_2$ are minimal, then both $y_1(t)$ and $y_2(t)$ are constant.
The fact that
\[ \lim_{t \, \rightarrow \, -\infty} \| y_1(t) - y_2(t) \|^2 =0 \]
shows that $M_1$ and $M_2$ must be the same submanifold. This proves part (2) of Theorem~\ref{bo}
and also completes the proof of Theorem~\ref{bo}.

\section{Norm of shape operators and mean curvature vectors}
\label{sec:norm}

Squared norms of shape operators and mean curvature vector fields are important geometric quantities in studying
ancient solutions of MCF (see, for example, \cite{HS}, \cite{AB} and \cite{LN}).
In this section we will compute these quantities for MCF of compact isoparametric submanifolds $M^n \subset \mathbb{R}^N$.
Without loss of generalities, we assume $M$ is contained in $S^{N-1}(r)$, the sphere of radius $r$ centered at the
origin in $\mathbb{R}^N$, for some $r>0$.  $S^{N-1}(1)$ will be abbreviated as $S^{N-1}$.
We will use $A^E$ and $H^E$ (respectively $A^S$ and $H^S$) denote the shape operator
and the mean curvature vector of $M$ as a submanifold of the Euclidean space $\mathbb{R}^N$
(respectively of the sphere $S^{N-1}(r)$).
Since MCF of $M$ at each point $x_0 \in M$ stays in the Weyl chamber in the normal space $\nu_{x_0}M $ containing $x_0$,
we will compute norm squares of shape operators and mean curvature vectors as functions on the Weyl chamber.

Fix an arbitrary point $x_0 \in M$ and let $\vn_i$, $i=1, \ldots, g$, be curvature normals of $M$ at the point $x_0$.
The assumption that $M$ is contained in a sphere centered at the origin implies
\[ <x_0, \vn_i>= -1 \]
for all $i$. Let $C$ be the open Weyl chamber in $\nu_{x_0}M$ containing $x_0$. Then $C$ is given by
equation \eqref{eqn:chamber}. For any $x \in C$, the parallel translation of $x-x_0$ defines
a parallel normal vector field of $M$, whose image under the exponential map in $\mathbb{R}^N$ gives
an isoparametric submanifold $M_x$ which is parallel to $M$. Moreover $C \cap M_x = \{x\}$ and
$\nu_x M_x = \nu_{x_0} M$.
Curvature normals of $M_x$ at the point $x$ are given by
\begin{equation} \label{eqn:curvnp}
 - \frac{\vn_i}{<x, \vn_i>}
\end{equation}
with multiplicity $m_i$ which is independent of x for $i=1, \ldots, g$.
Let
\[ \alpha_i := - \frac{\vn_i}{\|\vn_i\|}. \]
Then $\{\alpha_i \mid i=1, \ldots, g\}$ is the set of all {\it positive roots} of the Coxeter group $W$ with respect
to the choice of Weyl chamber $C$ in the sense of \cite{GB}.
Note that, unlike curvature normals, $\alpha_i$ are uniquely determined by the
Weyl chamber $C$ and are independent of the choice of points in $C$.
Therefore it will be convenient to describe all geometric quantities in terms of $\alpha_i$.
For example, $C$ is given by
\begin{equation} \label{eqn:ChamRoots}
C=\{ x \in \nu_{x_0}M \mid \li x, \alpha_i \ri \,\, > \,\,0 \,\,\, for \,\,\, all \,\,\, i \}.
\end{equation}
The Coxeter group $W$ is generated by reflections through hyperplanes perpendicular to $\alpha_i$
for $i=1, \ldots, g$.

\begin{lem} \label{lem:norm}
Let $A^E(x)$ and $H^E(x)$ (respectively $A^S(x)$ and $H^S(x)$) be the shape operator and mean curvature vector of $M_x$ at $x$
when considering $M_x$ as a submanifold of $\mathbb{R}^N$ (respectively of $S^{N-1}(\|x\|)$).
Then
\begin{eqnarray*}
H^E(x) &=& - \sum_{i=1}^g \frac{m_i \alpha_i}{\li x, \alpha_i \ri}, \\
H^S(x) &=& - \sum_{i=1}^g \frac{m_i \alpha_i}{\li x, \alpha_i \ri} + \frac{n x}{\|x\|^2}, \\
\|H^S(x)\|^2 &=& \|H^E(x)\|^2-\frac{n^2}{\|x\|^2},\\
\|A^E(x)\|^2 &=& \sum_{i=1}^g  \frac{m_i}{\li x, \alpha_i \ri^2}, \\
\|A^S(x)\|^2 &=& \sum_{i=1}^g  \frac{m_i}{\li x, \alpha_i \ri^2} - \frac{n}{\|x\|^2}.
\end{eqnarray*}
\end{lem}

\begin{proof} By formula \eqref{eqn:curvnp}, curvature normals of $M_x$ at $x$ is given by
\begin{equation} \label{eqn:curvroot}
 - \frac{\alpha_i}{<x, \alpha_i>}
\end{equation}
for $i=1, \ldots, g$.
Multiplying this vector by $m_i$ and summing over all $i$, we obtain $H^E(x)$.
$H^S(x)$ is obtained from $H^E(x)$ by subtracting its projection along the radial direction which
is equal to $- \frac{nx}{\|x\|^2}$. Since
\[ H^E(x)= H^S(x) - \frac{nx}{\|x\|^2} \]
is an orthogonal decomposition,
\[ \|H^E(x)\|^2 = \|H^S(x)\|^2 + \frac{n^2}{\|x\|^2}.\]

Let $\{\xi_j \mid j=1, \ldots, N-n\}$ be an orthonormal basis of $\nu_{x}M_x$.
Then the norm square of the Euclidean shape operator $A^E(x)$ of $M_x$ at $x$ is given by
\[ \|A^E(x)\|^2 = \sum_{j=1}^{N-n} \|A^E_{\xi_j}\|^2(x).\]
By formula \eqref{eqn:curvnorm},
\[ \|A^E(x)\|^2 = \sum_{j=1}^{N-n} \sum_{i=1}^g m_i \frac{\li \xi_j, \alpha_i \ri^2}{\li x, \alpha_i \ri^2}. \]
Since $\alpha_i$ are unit vectors in $\nu_{x}M_x$, $\sum_{j=1}^{N-n} \li \xi_j, \alpha_i \ri^2 = 1$. Hence
\begin{equation} \label{eqn:AE2}
\|A^E(x)\|^2= \sum_{i=1}^g  \frac{m_i}{\li x, \alpha_i \ri^2}.
\end{equation}
We can take $\xi_1=\frac{x}{\|x\|}$. Then $\{\xi_2, \ldots, \xi_{N-n}\}$ is an orthonormal basis of the normal
space of $M$ in the sphere $S^{N-1}(\|x\|)$ at point $x$.
So the norm square of the spherical shape operator is given by
\[ \|A^S(x)\|^2 =  \|A^E(x)\|^2 - \sum_{i=1}^g m_i \frac{\li \xi_1, \alpha_i \ri^2}{\li x, \alpha_i \ri^2}
= \|A^E(x)\|^2 - \frac{n}{\|x\|^2}. \]
The formula for $|A^S(x)|^2$ then follows from equation \eqref{eqn:AE2}.
This completes the proof of the lemma.
\end{proof}

Let $z \in C$ be a unit vector such that $M_z$ is a minimal isoparametric  submanifold in $S^{N-1}$. By part (2) of Theorem \ref{bo}, such $z$ is unique. Since $H^S(z)=0$, we have
\[ H^E(z)=-\sum_{i=1}^g \frac{m_i \alpha_i}{\li z, \alpha_i \ri} = -n z, \hspace{20pt} \,\,\, \|H^E(z)\|=n. \]

Let $x_0$ be an arbitrary unit vector in $C$ and
$x(t) \in C$ be the solution to the Euclidean mean curvature flow with $x(0)=x_0$. Then by equation~\eqref{eqn:limE},
\[ \lim_{t \rightarrow -\infty} \frac{x(t)}{\|x(t)\|} = z.\]
By equation \eqref{eqn:aa}, $\|x(t)\| = \sqrt{1-2nt}$. So an immediate consequence of Lemma~\ref{lem:norm} is the following
\begin{cor}
\begin{eqnarray*}
 \lim_{t \rightarrow -\infty} (1-2nt) \, \|H^E(x(t))\|^2 &=& n^2, \\
 \lim_{t \rightarrow -\infty} (1-2nt) \, \|A^E(x(t))\|^2 &=& \sum_{i=1}^g \frac{m_i}{\li z, \alpha_i \ri^2}.
\end{eqnarray*}
\end{cor}

Let $y(t)$ be the spherical MCF with $y(0)=x_0$. By Theorem~\ref{bo},
\[ \lim_{t \rightarrow -\infty} y(t) = z. \]
 So by Lemma~\ref{lem:norm},
we have
\begin{cor}
\begin{eqnarray*}
 \lim_{t \rightarrow -\infty} \|H^S(y(t))\|^2 &=& 0, \\
 \lim_{t \rightarrow -\infty} \|A^S(y(t))\|^2 &=& \sum_{i=1}^g \frac{m_i}{\li z, \alpha_i \ri^2} -n.
\end{eqnarray*}
\end{cor}

\section{MCF of Rank 2 Isoparametric Submanifolds}
\label{sec:rank2}

In \cite{LT}, the authors showed that explicit solutions to MCF of isoparametric submanifolds can be constructed
using invariant polynomials of Coxeter groups. However, in the rank 2 case, solutions can be solved directly. In this section we write down explicit solutions for the rank 2 case. We also compute the norms of the shape operators and mean curvature of these flow solutions, and explain their relations to rigidity problems and the Chern's conjecture on the norm of the second fundamental forms of minimal hypersurfaces in spheres.

Let $M$ be a compact $n$-dimensional isoparametric hypersurface in the unit sphere $S^{n+1}$ with $g$ distinct
principal curvatures.
Then $M$ is also a rank 2 isoparametric submanifold in $\mathbb{R}^{n+2}$.
We refer readers to a survey paper by Q-S Chi \cite{Chi} on the classifciation of isoparametric hypersurfaces in spheres.

Fix $x_0 \in M$ and let $\nu_{x_0}M$ be the normal space of $M$ as a submanifold of $\mathbb{R}^{n+2}$. The normal geodesic for $M$
at $x_0$ in $S^{n+1}$ is the unit circle  in $\nu_{x_0}M$ which is centered at origin of $\mathbb{R}^{n+2}$.
$M$ has two focal submanifolds, denoted by $M_+$ and $M_-$. We may assume $\dim(M_+)\leq \dim(M_-)$. It is known that
$M_+\cup M_-$ intersects  the normal geodesic at $x_0$ in exactly $2g$ points,
evenly distributed along the circle.
Let $x_{\pm}$ be the intersection of $M_{\pm}$ with the normal geodesic which are closest to $x_0$.
We may identify $\nu_{x_0}M$ with $\mathbb{C}$ such that $x_+ = 1 \in \mathbb{C}$ and
$x_- = e^{i\pi/g} \in \mathbb{C}$.

The Coxeter group $W$ of $M$ is the dihedral group of $2g$ elements acting on $\nu_{x_0}M \cong \mathbb{C}$.
The open Weyl chamber $C$ containing $x_0=e^{i \theta_0}$ is given by
\beq\label{aj}
 C = \{ r e^{i \theta} \,  \mid \, r>0, \,\, 0 < \theta < \pi/g\}.
 \eeq
In fact, $W$ is generated by reflections in the lines $\o=0$ and $\o=\pi/g$.
The set of positive roots is given by
\[ \{ \alpha_k :=e^{i \theta_k} \mid k=1, \ldots, g \} \]
where
\[ \theta_k := \frac{k\pi}{g}  - \frac{\pi}{2}. \]
Note that there is a mistake in \cite[Example 3.4]{LT} for the formula of positive roots, but this does not affect the
answer if we choose the Weyl chamber $C$ as above.

Let $m_k$ be the multiplicity of the curvature normal of $M$ at $x_0$ which is parallel to $\alpha_k$.
These are exactly the multiplicities of the principal curvatures of $M$ as a hypersurface in $S^{n+1}$.
It was proved by M\"{u}nzner in \cite{Mu} that the number of distinct principal curvatures of $M$
must be $g=1, 2, 3, 4, $ or $6$, and
\[ m_i = m_{i+2}  \]
for all $i$, where the subscript for $m_i$ is understood as $i \mod g$  if $i > g$.
Hence there are only two possible multiplicities $m_1$ and $m_2$. We will call $M$ an isoparametric hypersurface with $g$
distinct principal curvatures with {\it multiplicity data\/} $(m_1, m_2)$.
Abresch proved in \cite{Abr} that if $g=6$, then $m_1=m_2\in \{1,2\}$.
Hence for $g=1, 3, 6$, we must have $m_1= m_2$. For $g=2$ or $4$,  $m_1$ and $m_2$ may or may not be the same.
The assumption $\dim(M_+)\leq \dim(M_-)$ implies that $m_1 \leq m_2$.

Henceforth in this section we assume the following:
\ben
\item[(a)] $M^n$ is an isoparametric hypersurface in the unit sphere $S^{n+1}$ with $g$ distinct principal curvatures and multiplicity data $(m_1, m_2)$ with $m_1\leq m_2$.
Set
\begin{equation} \label{eqn:delta}
\delta=\delta(m_1, m_2) := \left\{ \begin{array}{ll} \frac{m_2-m_1}{m_2+m_1}, & {\rm if} \,\,\, g \geq 2, \\
                                    0,  & {\rm if} \,\,\, g=1.
                  \end{array} \right.
\end{equation}
Then $0\leq \delta <1$, and $\delta$ is possibly non-zero only if $g=2$ or $4$.
\item[(b)]
Fix $x_0\in M$, we identify $\nu_{x_0}M$ as $\C$ so that the focal submanifolds $M_+$ and $M_-$ are parallel
to $M$ through $1$ and $e^{i\pi/g}$ respectively as explained before.
\item[(c)] Let $C\subset \nu_{x_0} M$ be the Weyl chamber defined by \eqref{aj}. Given $x\in C$, let  $M_x$ denote the isoparametric hypersurface
 parallel to $M$ as submanifolds in $\R^{n+2}$.
\een

\begin{lem}\label{ae} Let $M, C, g, m_1, m_2, \d= \d(m_1,m_2)$ be as above.
If $x=re^{i \theta} \in C$, then the Euclidean and spherical mean curvature vectors for the isoparametric submanifold $M_x$ parallel to $M$ at  $x$
are given by
\begin{equation} \label{eqn:HE2}
H^E(x)= - \frac{n}{r} e^{i \theta} \left\{ 1+ i \cot g \theta + i \delta \csc g \theta  \right\}
\end{equation}
and
\begin{equation} \label{eqn:HS2}
H^S(x)= - \frac{n}{r} i e^{i \theta} \left\{  \cot g \theta + \delta \csc g \theta  \right\}.
\end{equation}
In particular, let $\o_{\rm min}$ be defined by
\begin{equation} \label{eqn:min}
\cos g \theta_{\rm min} = - \delta, \hspace{20pt} 0 < \theta_{\rm min} < \frac{\pi}{g}.
\end{equation}
  Then the isoparametric hypesurface through $e^{i\o_{\rm min}}$ is minimal.
\end{lem}

\begin{proof}
Note that
\begin{equation} \label{eqn:inner}
 \li x, \alpha_k \ri = r \cos(\theta-\theta_k).
\end{equation}
By Lemma \ref{lem:norm}, the Euclidean mean curvature vector $H^E$ for the isoparametric submanifold $M_x$ at $x$
is given by
\begin{eqnarray}
H^E(x) &=& - \sum_{k=1}^{g} \frac{m_k}{r \cos(\theta-\theta_k)} e^{i \theta_k} \nonumber \\
&=& - \frac{e^{i \theta}}{r}  \sum_{k=1}^{g} m_k \left\{ 1+i \tan(\theta_k-\theta) \right\}  \nonumber \\
&=& - \frac{e^{i \theta}}{r}   \left\{ n - i \sum_{k=1}^{g} m_k \cot(\frac{k \pi}{g}-\theta) \right\}  \nonumber \\
&=& - \frac{e^{i \theta}}{r} \left\{n- i  m_1 \Phi_g(\theta) -i (m_2-m_1) \Psi_g(\theta) \right\}. \label{eqn:1neq2}
\end{eqnarray}
where
\[
\Phi_g(\theta) := \sum_{k=1}^{g}  \cot(\frac{k \pi}{g}-\theta), \hspace{20pt}
 \Psi_g(\theta) := \sum_{j=1}^{\lfloor g/2 \rfloor} \cot(\frac{2j \pi}{g}-\theta). \]
By an elementary trigonometric identity
\begin{equation} \label{eqn:sumcot}
\sum_{k=1}^{g} \cot(\frac{k\pi}{g}+ \beta) \, = \, g \cot g \beta
\end{equation}
for any angle $\beta$ and any positive integer $g$, we have
\begin{eqnarray*}
\Phi_g(\theta)
&=& - g \cot g \theta
\end{eqnarray*}
for all positive integer $g$ and
\[ \Psi_g(\theta) = - \frac{g}{2} \cot \frac{g}{2} \theta = - \frac{g}{2} \left( \cot g \theta + \csc g \theta \right) \]
if $g$ is an even positive integer.

Since $m_1=m_2$ if $g$ is odd, by equation \eqref{eqn:1neq2}, we have
\begin{eqnarray*}
H^E(x)
&=& - \frac{e^{i \theta}}{r} \left\{n + i m_1 g \cot g \theta
            + i (m_2-m_1) \frac{g}{2}  \left( \cot g \theta + \csc g \theta \right) \right\}. \\
&=& - \frac{e^{i \theta}}{r} \left\{n + i n \cot g \theta
            + i  n \delta  \csc g \theta  \right\}
\end{eqnarray*}
for all $g$.
The last equality follows from the fact that
\begin{equation} \label{eqn:dim}
(m_1+m_2)g=2n.
\end{equation}
This proves equation \ref{eqn:HE2}. Equation \ref{eqn:HS2} then follows from equation \ref{eqn:HE2} and Lemma \ref{lem:norm}.
\end{proof}

Next we  give explicit solutions to the MCF for rank 2 isoparametric submanifolds.

\bprop \label{lem:MCF2}
Let $x(t)=r(t) e^{i \theta(t)}  \in C $ be the solution  for the Euclidean MCF of $M$ with
 $ x(0)=e^{i \theta_0} $. Then
\begin{align}
r(t) &= \sqrt{1-2nt}, \label{ag1} \\
\cos g \theta(t) &= (1-2nt)^{-\frac{g}{2}}  \left\{ \cos g\theta_0 + \delta \right\}-\delta. \label{ag2}
\end{align}
\eprop

\begin{proof}
By equation \eqref{eqn:HE2}, the Euclidean MCF $x'(t)=H^E(x)$ written in terms of $r$ and $\o$ becomes
$$r'(t)e^{i\o}+ i\o' r e^{i\o}= -\frac{n}{r} e^{i \theta} (1+ i (\cot g\o +\d \csc g\o)).$$
Note that $e^{i \theta}$ and $i e^{i \theta}$ are perpendicular vectors in $\mathbb{C}$, we have
$$\bca r'= -\frac{n}{r},\\
\o'= -\frac{n}{r^2} (\cot g\o+ \d \csc g\o).\eca$$
Integrate directly to get the solution.
\end{proof}

\bprop \label{cor:SMCF2}
Let $y(t)= e^{i \theta(t)}  \in C $ be the solution  for the spherical MCF of $M$ with
 $ y(0)=e^{i \theta_0} $.  Then
\beq\label{al}
\cos g \theta(t) = e^{gnt} \left\{ \cos g\theta_0 + \delta \right\}-\delta .
\eeq
\eprop

\begin{proof}
By equation \eqref{eqn:HS2}, we can write the spherical MCF $y'(t)= H^S(y)$ in terms of $\o$ and obtain
 $$\o'= -n (\cot g\o +\d \csc g\o).$$
 Integrate directly to get the solution.
\end{proof}

\ms
By Propositions \ref{lem:MCF2} and \ref{cor:SMCF2}, for both Euclidean and spherical MCF of isoparametric hypersurfaces, we have
 \beq\label{ak}
 \lim_{t \rightarrow - \infty} \theta(t) = \theta_{\rm min}.
 \eeq
 Recall that $\o_{\min}$ is defined by \eqref{eqn:min} and the isoparametric hypersurface in shpere through $e^{i\o_{\min}}$ is minimal. So the spherical MCF for isoparametric hypersurfaces tends to the minimal one as $t\to -\infty$.
 Moreover,
 \ben
\item  if $\o_0\in (0, \o_{\min})$, then the spherical MCF collapses to the focal submanifold $M_+$
which passes through $1$ as $t\to T_+$,
\item if $\o_0 \in (\o_{\min}, \frac{\pi}{g})$, then the spherical MCF collapses to the focal submanifold $M_-$ which passes
through $e^{i\pi/g}$ as $t\to T_-$,
\een
where $T_\pm= \frac{1}{gn} \ln \frac{\d\pm 1}{\d+\cos g\o_0}$ and $\d=\d(m_1, m_2)$.

\ms

\begin{eg} \label{ex:torus}
 Isoparametric hypersurfaces in $S^{n+1}$ with two distinct principal curvatures are embedded tori
 \[ T(\o) := S^k(\cos\o)\times S^{n-k}(\sin\o)\]
  with $0 < \o < \frac{\pi}{2}$.
  They have principal curvatures $-\tan\o$ and $\cot\o$
   with multiplicities $k$ and $n-k$ respectively. Let $\o_\min\in (0, \frac{\pi}{2})$ satisfying
 \[ \cos 2\o_\min=-\d= -\frac{n-2k}{n}.\]
  Then
 \[ \cos\o_\min= \sqrt{\frac{n-k}{n}}, \hspace{20pt} \sin\o_\min= \sqrt{\frac{k}{n}}\]
  and $T(\o_\min)$ is the well-known minimal Clifford torus. The  spherical MCF with initial data $T(\o_0)$ is $T(\o(t))$, where
 $$\cos 2\o(t)= e^{2nt} (\d+ \cos 2\o_0) -\d.$$
 Note that:
 \ben
 \item $\lim_{t\to-\infty}\o(t)= \o_\min$, so the flow tends to the minimal torus.
\item If $\o_0\in (0, \o_\min)$, then the flow collapses to $S^k(1)\times 0$ as $$t\to \frac{1}{2n}\ln \frac{\d+1}{\d+\cos2\o_0}.$$
\item If $\o_0\in (\o_\min, \frac{\pi}{2})$, then the flow collapses to $0\times S^{n-k}(1)$ as $$t\to   \frac{1}{2n}\ln \frac{\d-1}{\d+\cos2\o_0}.$$
\een
\end{eg}

\begin{eg} \label{ex:g3n3}
  The simplest isoparametric hypersurfaces in the sphere with three distinct principal curvatures are those with uniform multiplicity $1$.
 They are the embeddings of manifold of flags of $\R^3$ in $S^4$ as principal orbits  of the conjugation action of $SO(3)$ on the space of trace zero real symmetric $3 \times 3$
matrices.  Each such orbit contains a diagonal matrix $\bc =\diag(c_1, c_2, c_3)$, where $c_1, c_2, c_3$ are distinct real numbers with
$$c_1+c_2+c_3=0, \hspace{20pt} c_1^2+c_2^2+ c_3^2=1.$$
We denote this orbit by $M_\bc$.
 The minimal isoparametric hypersurface in this family is the principal orbit $M_{\bc_{\min}}$ where
  $$\bc_{\min}= \diag(\frac{1}{\sqrt{2}}, 0, -\frac{1}{\sqrt{2}}).$$
Two focal submanifolds are the singular orbits $M_{\bc_+}$ and  $M_{\bc_-}$ where
\[ \bc_+ = \frac{1}{\sqrt{6}}\diag(1,1,-2),  \hspace{20pt} \bc_- = \frac{1}{\sqrt{6}} \diag(2,-1,-1). \]
Identify the normal plane of $M_\bc$ at $\bc$ as $\C$ with $\bc_+$ identified as $1$ and $\frac{1}{\sqrt{2}}\diag(1, -1,0)$ as $i= e^{i\frac{\pi}{2}}$. Under this identification
$\bc_{\min}$ is identified with $e^{i\frac{\pi}{6}}$ and $\bc_-$ is identified with $e^{i \frac{\pi}{3}}$. By \eqref{al}, the spherical MCF for this family is $M_{\bc(t)}$ with
$$\bc(t)= \cos \o(t) \frac{1}{\sqrt{6}}\diag(1,1,-2) + \sin \o(t) \frac{1}{\sqrt{2}}\diag(1, -1,0),$$
where  $$\cos 3\o(t)= e^{3nt} \cos 3\o(0).$$
So we have
\ben
\item $\lim_{t\to -\infty} M_{\bc(t)}= M_{\bc_\min}$, the minimal isoparametric hypersurface.
\item If $\o(0)\in (0, \frac{\pi}{6})$, then $M_t$ collapses to the singular orbit $M_{\bc_+}$ as $t\to \frac{1}{3n}\ln \frac{ 1}{\cos3\o(0)}$.
\item If $\o(0)\in (\frac{\pi}{6}, \frac{\pi}{3})$, then the flow collapses to the singular orbit  $M_{\bc_-}$
  as $t\to \frac{1}{3n}\ln \frac{- 1}{\cos3\o(0)}$.
\een
Note that both $M_{\bc_+}$ and $M_{\bc_-}$ are embeddings of $\R P^2$ in $S^4$, classically known as the Veronese embeddings of the real projective plane in $S^4$.
\end{eg}

\ms

Next we compute the norm of shape operators of isoparametric hypersurfaces. First we need the following elementary identity (we
include a proof here for completeness).
\begin{lem} \label{lem:sumcot2}
\[ \sum_{k=1}^{g} \cot^2(\frac{k\pi}{g}+\beta) = g^2 \csc^2(g\beta) - g\]
for any angle $\beta$ and any positive integer $g$.
\end{lem}

\begin{proof}
By the well known identity
\[ \cot g \beta = \frac{\sum_{k \,\,\, {\rm even}} (-1)^{k/2} \left(\begin{array}{c} g \\ k \end{array} \right) \cot^{g-k} \beta}
                        {\sum_{k \,\,\, {\rm odd}} (-1)^{(k-1)/2} \left(\begin{array}{c} g \\ k \end{array} \right) \cot^{g-k} \beta}, \]
we know that
$y_k = \cot(\frac{k\pi}{g}+\beta)$, $k=1, \ldots, g$, are solutions of the following degree $g$ polynomial eqation for $y$:
\[ \sum_{k \,\,\, {\rm even}} (-1)^{k/2} \left(\begin{array}{c} g \\ k \end{array} \right) y^{g-k}
- \cot g \beta \sum_{k \,\,\, {\rm odd}} (-1)^{(k-1)/2} \left(\begin{array}{c} g \\ k \end{array} \right) y^{g-k}  = 0.\]
Hence the formula in this lemma, as well as formula \eqref{eqn:sumcot}, can be proved using
coefficients of $y^{g-1}$ and $y^{g-2}$ in the above equation.
\end{proof}

\begin{lem} \label{lem:norm2}
For $x=r e^{i \theta} \in C$, the norm square of Euclidean and spherical shape operators of $M_x$ at $x$
are given by
\begin{align}
 \| A^E(x) \|^2 &= \frac{ng}{r^2 } \csc^2 g \theta \left( 1 + \delta \cos g \theta  \right), \label{aq1} \\
 \| A^S(x) \|^2 &= \frac{ng}{r^2}  \csc^2 g \theta \left(1 + \delta \cos g \theta  \right)-\frac{n}{r^2}. \label{aq2}
\end{align}
\end{lem}

\begin{proof}
By Lemma \ref{lem:norm} and equation \eqref{eqn:inner}, we have
\begin{eqnarray*}
\|A^E(x)\|^2 & =& \sum_{k=1}^g \frac{m_k}{r^2 \cos^2(\theta-\theta_k)} \\
&=& \sum_{k=1}^g \frac{m_k}{r^2} \left\{1+ \cot^2(\frac{k\pi}{g}-\theta) \right\} \\
&=& \frac{n}{r^2}+\frac{m_1}{r^2} \sum_{k=1}^g  \cot^2(\frac{k\pi}{g}-\theta)
    +\frac{m_2-m_1}{r^2} \sum_{j=1}^{\lfloor g/2 \rfloor}  \cot^2(\frac{2j\pi}{g}-\theta).
\end{eqnarray*}

By Lemma \ref{lem:sumcot2}, we have
\begin{eqnarray*}
\|A^E(x)\|^2
&=& \frac{n}{r^2}+\frac{m_1}{r^2} \left( g^2 \csc^2 g\theta-g \right)
    +\frac{m_2-m_1}{r^2} \left( \frac{g^2}{4} \csc^2 \frac{g}{2}\theta-\frac{g}{2} \right).
\end{eqnarray*}

The Lemma then follows from half angle formula
\[ \csc^2 \frac{g}{2}\theta = 2 \csc^2 g \theta + 2 \cot g\theta \csc g \theta,\]
equation \eqref{eqn:dim}, and Lemma \ref{lem:norm}.
$\Box$

\begin{cor} \label{cor:normmin2}
If $M_x$ is minimal in the sphere, then
\[  \| A^E(x) \|^2 = \frac{ng}{r^2}, \hspace{20pt} \|A^S(x) \|^2 = \frac{n}{r^2} \left( g-1 \right). \]
\end{cor}
{\bf Proof}: By equation \eqref{eqn:min}, if $M_x$ is minimal in the sphere, then $x=r e^{i \theta}$ with $\theta$ satisfying the condition
\[ \delta = - \cos g \theta. \]
So
\[
1 + \delta \cos g \theta =  1 - \cos^2 g \theta = \sin^2 g \theta.
\]
The Corollary then follows from Lemma \ref{lem:norm2}.
\end{proof}

\brem  It is known (cf. \cite{Mu}) that principal curvatures of an isoparametric hypersurface $M^n$ in $S^{n+1}$ are $\{\cot(\o+\frac{(i-1)\pi}{g}) \n 1\leq i\leq g\}$ with multiplicity  data $(m_1, m_2)$.  We can also compute $||A^S||^2$ and $|| H^S||^2$ using this result and get the same result as before (with a suitable assignment of multiplicity to each principal curvature).
 For example, the formula of  $||A^S||^2$ given in Corallary \ref{cor:normmin2}  was obtained this way in \cite{PT1}.
\erem

\ms
\brem \label{rem:sharp3}
It was mentioned in the end of \cite{HS} that the condition \eqref{at} is sharp because for the torus $S^1(\cos\o) \times S^{n-1}(\sin \o)\subset S^{n+1}$, the quantity $||A^S||^2-\frac{1}{n-1}||H^S||^2 - 2$ can be arbitrarily close to $0$ as $\o\to 0$. This is the case of $g=2$ isoparametric hypersurfaces with principal curvatures $-\tan\o$, $\cot\o$ of multiplicities $1$ and $(n-1)$ respectively. A simple computation implies that
\beq\label{au}
||A^S||^2-\frac{1}{n-1} ||H^S||^2-2= \frac{n-2}{n-1} \tan^2\o
\eeq
(cf. also \cite{H}).
We have seen in Example \ref{ex:torus} that this ancient solution of spherical MCF has the following properties: $\o(t)\to 0$ as $t\to T$ for some finite $T>0$ and $\o(t)\to \o_{\min}$ as $t\to -\infty$, where $\cos 2\o_{\min}= -\d= -\frac{n-2}{n}$.  So $\tan^2\o_\min= n-1$.
 Note that the right hand side of \eqref{au} is arbitrarily small when $t\to T$, but is not arbitrarily small (it tends to $n-2$) as $t\to -\infty$. Hence this example does not justify the sharpness of condition \eqref{at} in Huisken-Sinestrari's Theorem mentioned in the introduction.
\erem

To study the sharpness of condition \eqref{ai} in Huisken-Sinestrari's Theorem, we need the following results:
\begin{lem}
For $x=r e^{i \theta} \in C$,
\begin{eqnarray} \label{eqn:AH=d}
 && \|A^S(x)\|^2 - \frac{g}{2n} \| H^S(x)\|^2 \nonumber \\
 &=&  \frac{n}{2r^2} \left\{ g (1-\delta^2) \csc^2 g \theta + (g-2) \right\}.
\end{eqnarray}
In case that all multiplicities of the isoparametric submanifolds are the same, which is always true when $g=1, 3, 6$, we can get a simpler
formula
\begin{eqnarray} \label{eqn:AH=}
\|A^S(x)\|^2 - \frac{g}{n} \| H^S(x)\|^2 &=&  \frac{n(g-1)}{r^2}.
\end{eqnarray}
\end{lem}

\begin{proof}
By Lemma \ref{lem:norm2} and equation \eqref{eqn:HS2}, we have
\begin{eqnarray*}
 && \frac{r^2}{n} \left\{ \|A^S(x)\|^2 - \frac{g}{2n} \| H^S(x)\|^2 \right\} \\
 &=&   g (1+ \delta \cos g \theta) \csc^2 g \theta - 1  - \frac{g}{2} (\delta + \cos g \theta)^2 \csc^2 g \theta \\
 &=&    \frac{g}{2} (2 -\delta^2 -\cos^2 g \theta ) \csc^2 g \theta - 1 \\
 &=&    \frac{g}{2} (1 -\delta^2 + \sin^2 g \theta ) \csc^2 g \theta - 1 \\
 &=&    \frac{g}{2} (1 -\delta^2) \csc^2 g \theta + \frac{g}{2}- 1.
\end{eqnarray*}
This proves equation \eqref{eqn:AH=d}.

If all multiplicities are the same, then $\delta=0$.
By Lemma \ref{lem:norm2},
\[  \| A^S(x))\|^2 =  \frac{n}{r^2} (g \cot^2 g \theta+g -1) \]
for $x=r e^{i \theta} \in C$.
This implies equation \eqref{eqn:AH=}  since in this case
\[   \cot^2 g \theta = \frac{r^2}{n^2} \| H^S(x))\|^2 \]
by equation \eqref{eqn:HS2}.
\end{proof}

\begin{cor}
If $y(t)=e^{i \theta(t)} \in C$ is a solution to the spherical MCF with $\theta(0) = \theta_0 \neq \theta_{\rm min}$, then
\begin{eqnarray} \label{eqn:AH<}
\|A^S(y(t))\|^2 & < & \frac{g(1+\delta)}{n (\cos g \theta_0 + \delta)^2} \, e^{-2gnt} \, \| H^S(y(t))\|^2.
\end{eqnarray}
\end{cor}

\begin{proof}
If the isoparametric submanifold passing through $x$ is not minimal in the sphere, then
by equation \eqref{eqn:AH=d},
\begin{eqnarray*}
 \frac{  \|A^S(y(t))\|^2}{\| H^S(y(t))\|^2}
 &=&  \frac{g}{2n}  + \frac{n}{2} \,\, \frac{g (1-\delta^2) \csc^2 g \theta(t) + g-2}{\| H^S(y(t))\|^2}
\end{eqnarray*}
By equation \eqref{eqn:HS2}, we have
\begin{eqnarray}
\frac{2n}{g} \,\,\, \frac{  \|A^S(y(t))\|^2}{  \| H^S(y(t))\|^2}
 &= &  1  + \frac{(1-\delta^2) \csc^2 g \theta(t) + 1-\frac{2}{g}}{(\cot g \theta(t)+ \delta \csc g \theta(t))^2}   \nonumber \\
 &= &  1  + \frac{(1-\delta^2)  + (1-\frac{2}{g})\sin^2 g \theta(t)}{(\cos g \theta(t)+ \delta )^2}   \nonumber \\
 &\leq &  1  + \frac{(1-\delta^2)  + \sin^2 g \theta(t)}{(\cos g \theta(t)+ \delta )^2}   \nonumber \\
 &= &   \frac{2 + 2 \delta \cos g \theta(t)}{(\cos g \theta(t)+ \delta )^2}   \nonumber \\
 &\leq &   \frac{2(1 + \delta)}{(\cos g \theta(t)+ \delta )^2}   \nonumber \\
 &= &   \frac{2(1 + \delta)}{(\cos g \theta_0+ \delta )^2} \,\,\, e^{-2gnt}.  \label{eqn:AH<=}
\end{eqnarray}
The last equality follows from Proposition \ref{cor:SMCF2}. Note that in the above estimates,
equality holds only if $\theta(t)=0$ or $\pi/g$, in which cases the corresponding ispoparametric hypersurface
collapses to a lower dimensional focal submanifold and therefore never occur for all $t$ where the MCF exists.
Hence inequality \eqref{eqn:AH<=}  implies inequality \eqref{eqn:AH<}.
\end{proof}

\brem\label{prop:sharp2}
Note that $\cos g \theta_0$ can be arbitrarily close to $1$ when $\theta_0$ tends to $0$.
So if $n > g$, we can easily choose $\theta_0$ such that
\[ \frac{g(1+\delta)}{n (\cos g \theta_0 + \delta)^2} < 1. \]
In these cases, inequality \eqref{eqn:AH<} implies
\begin{equation}
 \|A^S(t)\|^2  <   e^{-2gnt} \, \| H^S(t)\|^2
\end{equation}
for all $t$.
In particular, when $g=2$ and $n\geq 3$, we can choose $\o_0$ close to $0$ such that the ancient solutions
of the spherical MCF described in Example \ref{ex:torus} satisfy
\[ \|A^S(t)\|^2  <   e^{-4nt} \, \| H^S(t)\|^2 \]
for all $t$. Hence Condition \eqref{ai} in Huisken-Sinestrari's result mentioned in the introduction is sharp in the sense that
$B$ can not be replaced by a constant greater than or equal to $4n$.
\erem

In the rest part of this paper, we will study rigidity problems of ancient solutions of spherical MCF modeled
on isoparametric hypersurfaces. For simplicity, we will
use $A$ and $H$ denote the shape operator and the mean curvature of a hypersurface in $S^{n+1}$.
If $M_t$ is an ancient solution for the spherical MCF of hypersurfaces in the sphere,
we will use $A(t)$ and $H(t)$ to denote the shape operator and mean curvature of $M_t$.

\bthm\label{am}
 Let $M_t$ be an ancient MCF for isoparametric hypersurfaces in $S^{n+1}$ with $g$ distinct principal curvatures. Then
 \beq\label{an}
\lim_{t\to -\infty} ||A( t)||^2= (g-1)n.
\eeq
If $M_0$ is not minimal, then there is a positive constant $C_0$ only depending on $M_0$ such that
\begin{equation} \label{eqn:limitH}
\lim_{t\to -\infty} H^2(t) \,\, e^{-2gnt}= C_0.
\end{equation}
In fact $C_0$ is given by equation \eqref{eqn:A/Hconst}.
\ethm

\begin{proof}
It follows from \eqref{ak} and Corollary \ref{cor:normmin2}  that
$$\lim_{t\to -\infty} ||A(t)||^2= ||A(\o_\min)||^2=(g-1)n.$$
By equations \eqref{eqn:HS2} and  \eqref{al}, we have
 $$H^2(t)= \frac{n^2 (\cos g\o_0+\d)^2}{\sin^2g\o(t)}\, e^{2gnt},$$
where $\d$ is given by equation \eqref{eqn:delta}.
Since
$$\lim_{t\to -\infty} \o(t)= \o_{\min}, \hspace{20pt} {\rm and} \,\,\, \cos g\o_{\min}= -\d,$$
 we have
$$\lim_{t\to -\infty} \sin^2g\o(t)= 1-\d^2.$$
This implies equation \eqref{eqn:limitH} with
\begin{equation} \label{eqn:A/Hconst}
C_0=\frac{n^2(\cos g \theta_0 + \delta)^2}{(1-\delta^2)}.
\end{equation}
\end{proof}

\bcor\label{ar} For spherical MCF of isoparametric hypersurfaces, we have the following estimates:
\ben
\item For $g=1$, we have
$\frac{||A(t)||^2}{H^2(t)}\equiv \frac{1}{n}$.
\item For $g\geq 2$,  there exist $t_1>0$ and positive constants $c_1$ and $c_2$ such that
\beq\label{ao}
c_2e^{-2gnt} \leq \frac{||A(t)||^2}{H^2(t)}\leq c_1 e^{-2gnt}.
\eeq
for all $t< -t_1$.
\een
\ecor

\begin{proof}
Part (1) follows from equation \eqref{eqn:AH=}.
Part (2) follows from Theorem \ref{am} since
$$\lim_{t \rightarrow - \infty} \frac{||A(t)||^2}{H^2(t)} \,\, e^{2gnt}$$
is a positive constant when $g \geq 2$.
\end{proof}

The spherical MCF for isoparametric hypersurfaces with $g=1$ is the spherical cap solution.
Huisken-Sinestrari's Theorem is a rigidity result modeled on this example. Motivated by
condition \eqref{ai} in Huisken-Sinestrari's Theorem and the estimate in
Corollary \ref{ar} above, we would like to propose the following rigidity conjecture
modeled on spherical MCF of isoparametric hypersurfaces with $g\geq 2$ distinct principal curvatures.

\ms
\begin{conj}
\label{conj:A/H}
 Let $f(t,\cdot)$ be an ancient solution to the spherical MCF for smooth compact hypersurfaces in $S^{n+1}$,  $A(t,\cdot)$ and $H(t,\cdot)$ the shape operator and mean curvature of $f(t,\cdot)$. If $A(t,\cdot)$ and $H(t,\cdot)$ satisfy the inequality \eqref{ao}
for some $g\in \{2, 3, 4, 6\}$,  then $f(t,\cdot)$ is a spherical MCF for isoparametric hypersurfaces with $g$ distinct principal curvatures.
\end{conj}

To have a rigidity conjecture with condition similar to inequality \eqref{at},  we will consider the
norm square of the traceless part of the shape operator:
\beq\label{af}
\phi:= ||A-\frac{1}{n} H \I ||^2 = ||A||^2-\frac{1}{n} H^2,
\eeq
where $I$ is the identity operator. This quantity arises naturally in studying
gap theorem for hypersurfaces in the sphere with constant mean curvature (cf. \cite{AdC}).
We first give some estimates of $\phi(t)$ for the spherical MCF of isoparametric hypersurfaces.

\bprop\label{as} Consider the spherical MCF of isoparametric hypersurfaces with $g$ distinct principal curvatures and multiplicity data $m_1\leq m_2$. Let $\phi(t,\cdot)$ be defined by equation \eqref{af} and
$\d=\d(m_1, m_2)$ defined by equation \eqref{eqn:delta}.
\ben
\item If $\d=0$, then given any $0<\e<1$, there exists $c_0>0$ such that  if $\o_0\in (\frac{\pi}{2g}-c_0, \frac{\pi}{2g}+c_0)$, then
\beq\label{ap1}
(g-1)n\leq \phi(t,\cdot)\leq (g-1+\e)n.
\eeq
Note that in this case $\theta_{\rm min}= \frac{\pi}{2g}$.
\item If $\d>0$ (so $g=2$ or $4$), then given $0<\e<1$, there exists $c_0>0$ such that
\ben
\item if $\o_0\in (\o_{\min}-c_0, \o_{\min})$ then
\beq\label{ap2}
(g-1-\e)n\leq \phi(t,\cdot)\leq (g-1)n,
\eeq
\item if $\o_0\in (\o_{\min}, \o_{\min}+c_0)$ then
$(g-1)n\leq \phi(t, \cdot)\leq (g-1+\e)n$,
\een
where $\o_{\min}\in (0, \frac{\pi}{g})$ such that $\cos (g\o_{\min})= -\d$.
\een
\eprop

\begin{proof}
An immediate consequence of Theorem \ref{am}  is the following:
For $g\geq 2$,  given $1>\e>0$, there exists $t_0>0$ such that
\beq\label{ap}
(g-1-\e)n  \leq \phi(t) \leq  (g-1+\e) n
\eeq
for all $t< -t_0$. For $g=1$,  $\phi(t) \equiv 0$ by Corollary \ref{ar}.

Using equations \eqref{eqn:HS2} and \eqref{aq2} to compute directly, we have
\begin{equation} \label{eqn:phi}
\phi(t)= n \csc^2 g\o(t) (g-1-\d^2 + \d (g-2) \cos g\o(t)).
\end{equation}

If $\d=0$, then
\[ \phi(t)= (g-1)n \csc^2 g\o(t) \geq (g-1)n. \]
 Part (1) of this proposition then follows from equation \eqref{ap}.

If $\d>0$ and $g=2$, then we have $\phi(t)= n(1-\d^2) \csc^2(2\o(t))$. So part (2) of this proposition follows for $g=2$.

If $\d>0$ and $g=4$, then we have
$$\phi(t)= n (3-\d^2 + 2\d \cos 4\o(t)) \csc^2 4\o(t).$$
Write $\xi(t)= \cos 4\o(t)$, then $\phi(t)= f(\xi(t))n$, where
$$f(\xi)=\frac{3-\d^2 +2\d \xi}{1-\xi^2}.$$
Note that $f(-\d)= 3$ and $f'(-\d)<0$.
So there exists $\e>0$ such that $f$ is decreasing on the interval $(-\d -\e, -\d+\e)$. So part (2) follows for $g=4$.

\end{proof}

Based on the above estimates, we would like to propose the following rigidity conjecture:

\begin{conj}
\label{conj:ay}
 Let $f(t,\cdot)$ be an ancient solution to the spherical MCF for smooth compact hypersurfaces in $S^{n+1}$. If the shape operator and
 mean curvature of $f(t,\cdot)$ satisfy inequality \eqref{ap1} or inequality \eqref{ap2}
for some $g\in \{2, 3, 4, 6\}$,  then $f(t,\cdot)$ is a spherical MCF for isoparametric hypersurfaces with $g$ distinct principal curvatures.
\end{conj}

\begin{rem} \label{rem:weakconj}
Otsuki constructed in \cite{O} closed minimal hypersurfaces in $S^{n+1}$ with two distinct principal curvatures of multiplicities $1$ and $n-1$ respectively such that $n-c_0\leq ||A||^2\leq n+c_0$ for some $c_0$. This  indicates that inequalities \eqref{ap1} and \eqref{ap2} in Conjecture \ref{conj:ay} can not be replaced by inequality \eqref{ap}.
\end{rem}

\begin{rem} By equation \eqref{af}, $||A||^2 -\frac{1}{n} H^2 \geq 0$. Hence the analogue of
Conjecture \ref{conj:ay} for the $g=1$ case can be stated as follows: Let $f(t, \cdot)$ be an ancient MCF for compact hypersurfaces in $S^{n+1}$. If there exists $\e>0$ such that
\begin{equation} \label{eqn:g1conj2}
||A(t, \cdot)||^2 -\frac{1}{n} H(t, \cdot)^2 \leq n-\e
\end{equation}
for $t$ sufficiently negative, then $f(t, \cdot)$ is a shrinking spherical cap or an equator. In fact, for $n \geq 2$, this statement follows from the following result of Lei-Xu-Zhao (cf. Theorem 2 in \cite{LXZ}):
If
\begin{equation} \label{eqn:LXZ}
 \limsup_{t \rightarrow - \infty} \,\,
\max_{M_t} (\|A\|^2 - \kappa H^2) < n
\end{equation}
where $\kappa = \min \{ \frac{3}{n+2}, \frac{4(n-1)}{n(n+2)} \} $, then
$f(t, \cdot)$ is a shrinking spherical cap or an equator.
Note that $\kappa \geq \frac{1}{n}$ for $n \geq 2$. Hence inequality \eqref{eqn:g1conj2}
implies inequality \eqref{eqn:LXZ} if $n \geq 2$.
\end{rem}
 \ms

 Next we explain the relation between Conjecture \ref{conj:ay} and Chern's Conjecture on the norm of minimal hypersurfaces in spheres. Using results of J. Simon \cite{Si}, Chern-Do Carmo-Kobayashi proved in \cite{CDK} the following gap theorem: If $M^n$ is a compact minimal hypersurface in $S^{n+1}$ and $0\leq ||A(x)||^2\leq n$ for all $x \in M$, then either $M$ is an equator with $||A(x)||^2=0$ or is a Clifford torus with $||A(x)||^2=n$.

\ms
\ni {\bf Chern's conjecture:}
The set
\begin{eqnarray*}
\cs &:=& \{ S\,\n \,\, {\rm There \,\, exists \,\,  a \,\,  compact \,\, minimal \,\, hypersurface \,\,  } M \subset S^{n+1} \\
&& \hspace{20pt} {\rm \,\, with \,\, constant \,\,} ||A|| {\rm \,\, such \,\, that \,\,} ||A||^2 =S \}
\end{eqnarray*}
is discrete, where $A$ is the shape operator of $M$.

\ms
Below we give a brief review of some results concerning Chern's conjecture: Let $M$ be a minimal hypersurface in $S^{n+1}$, $A$ the shape operator of $M$, and $S(x):= ||A(x)||^2$ for $x \in M$.
\ben
\item Peng-Terng proved in \cite{PT1}, \cite{PT2} that
\ben
\item if $S(x)$ is a constant and $n\leq S\leq n+\frac{1}{12}n$, then $S=n$ and $M$ is a Clifford torus,
\item if $n\leq 5$ then there exists $0<c<1$ such that $n\leq S(x)\leq n+c n$ for all $x\in M$ implies $S(x)\equiv n$ and $M$ is a Clifford torus.
\een
\item Chang proved in \cite{Cha} that if $n=3$ and $S> 3$ is constant, then $M$ is isoparametric with three distinct principal curvatures and $S=6$ (this minimal hypersurface in $S^4$ is described in Example \ref{ex:g3n3}).
\item Ding and Xin proved in \cite{DX} that if
$n\leq S(x)\leq n+ \frac{n}{23}$, then $S(x)\equiv n$ and $M$ is a Clifford torus. Xu-Xu in \cite{XX} improved the constant $\frac{1}{23}$ to $\frac{1}{22}$.
\een
We refer readers to  Xu-Xu's paper \cite{XX} on a survey of Chern's conjecture.

Note that for the stationary spherical MCF of hypersurfaces in $S^{n+1}$, Conjecture \ref{conj:ay} can be stated as follows: Let $M^n$ be a minimal hypersurface of $S^{n+1}$, and $S(x)= ||A(x)||^2$. If $S(x)$ satisfies one of the following inequalities for some $0<\e<1$,
 \begin{align*}
 & (g-1)n\leq S(x)\leq (g-1+\e)n,\\
 & (g-1-\e)n\leq S(x)\leq (g-1)n,
 \end{align*}  for all $x\in M$, then $M$ is isoparametric with $g$ distinct principal curvatures. In particular, this implies that $(g-1)n$ is a discrete point of the set $\cs$. Known results concerning rigidity of the stationary case are for $g=1, 2$  and $g=3$ with $m_1=m_2=1$.
All these works used estimates obtained from elliptic equations for $\D \II$ and $\D \K \II$, where $\II$ is the second fundamental form. We wonder whether the flow (parabolic) method may provide new insights and techniques to prove Chern's conjecture.


\end{document}

%% file: header_nosubsection.tex
\theoremstyle{plain}
\newtheorem{thm}{Theorem}[section]
\newtheorem{prop}[thm]{Proposition}
\newtheorem{lem}[thm]{Lemma}
\newtheorem{cor}[thm]{Corollary}
\newtheorem{conj}[thm]{Conjecture}

\theoremstyle{definition}

\newtheorem{rem}[thm]{Remark}
\newtheorem{defn}[thm]{Definition}
\newtheorem{eg}[thm]{Example}
\newtheorem{subtitle}[thm]{}
\newtheorem{ex}{Exercise}[section]
\numberwithin{equation}{section}

\def\d{\delta}
\def\D{\triangle}
\def\e{\epsilon}

\def\K{\nabla}

\def\n{\vert\/}
\def\o{\theta}

\def\cs{{\mathcal{S}}}

\def\li{\langle}
\def\ri{\rangle}
\def\n{|\/ }

\def\bs{\bigskip}
\def\ms{\medskip}

\def\ni{\noindent}
\def\ti{\tilde}
\def\p{\partial}

\def\I{{\rm I\/}}
\def\II{{\rm II\/}}
\def\diag{{\rm diag}}

\def\C{\mathbb{C}}

\def\R{\mathbb{R} }

\newcommand{\beq}{\begin{equation}}
\newcommand{\eeq}{\end{equation}}
\newcommand{\beg}{\begin{eg}}
\newcommand{\eeg}{\end{eg}}
\newcommand{\bthm}{\begin{thm}}
\newcommand{\ethm}{\end{thm}}
\newcommand{\bprop}{\begin{prop}}
\newcommand{\eprop}{\end{prop}}
\newcommand{\bcor}{\begin{cor}}
\newcommand{\ecor}{\end{cor}}
\newcommand{\blem}{\begin{lem}}
\newcommand{\elem}{\end{lem}}
\newcommand{\bca}{\begin{cases}}
\newcommand{\eca}{\end{cases}}
\newcommand{\brem}{\begin{rem}}
\newcommand{\erem}{\end{rem}}
\newcommand{\bpm}{\begin{pmatrix}}
\newcommand{\epm}{\end{pmatrix}}
\newcommand{\bbm}{\begin{bmatrix}}
\newcommand{\ebm}{\end{bmatrix}}
\newcommand{\bvm}{\begin{vmatrix}}
\newcommand{\evm}{\end{vmatrix}}
\newcommand{\bdefn}{\begin{defn}}
\newcommand{\edefn}{\end{defn}}
\newcommand{\bsub}{\begin{subtitle}}
\newcommand{\esub}{\end{subtitle}}
\newcommand{\bex}{\begin{ex}}
\newcommand{\eex}{\end{ex}}
\newcommand{\ben}{\begin{enumerate}}
\newcommand{\een}{\end{enumerate}}

%% file: AMCF_isoparametric_v9.bbl
\begin{thebibliography}{399}

\bibitem[Abr]{Abr} U. Abresch,
    {\it Isoparametric hypersurfaces with four or six principal curvatures}, Math. Ann. 264 (1983), 283--302.

\bibitem[AdC]{AdC} H. Alencar and M. do Carmo,
    {\it Hypersurfaces with constant mean curvature in spheres}, Proc. of AMS, Volume 120 (1994), No. 4, 1223--1229.

\bibitem[AAAW]{AAAW} D. Altschuler, S. Altschuler, S. Angenent, and L. Wu,
    {\it The zoo of solitons for curve shortening in $\mathbb{R}^n$}, Nonlinearity, 26 (2013), 1189--1226.

\bibitem[AB]{AB} B. Andrews and C. Baker,
    {\it Mean curvature flow of pinched submanifolds to spheres},
    J. Differential Geom. 85 (2010), no. 3, 357--395.

\bibitem[A]{A} S. Angenent, {\it Shrinking doughnuts}, in Nonlinear diffusion equations and
their equilibrium states, (1989, Gregynog), Birkhauser, Boston (1992).

\bibitem[BLM]{BLM} T. Bourni, M. Langford, and A. Mramor,
    {\it On the construction of closed nonconvex nonsoliton ancient mean curvature flows}, arXiv:1911.05641.

\bibitem[BLT]{BLT} T. Bourni, M. Langford, and G. Tinaglia,
    {\it Collapsing ancient solutions of mean curvature flow}, arXiv:1705.06981.

\bibitem[Cha]{Cha} S.P. Chang, {\it On minimal hypersurfaces with constant scalar curvatures in} $S^4$, J. Differential Geom. 37 (1993) 523--534.

\bibitem[CDK]{CDK} S.S. Chern, M. do Carmo, S. Kobayashi, {\it Minimal submanifolds of a sphere with second fundamental form of constant length}, Functional Analysis and Related Fields, Springer-Verlag, Berlin, 1970, 59--75.

\bibitem[Chi]{Chi} Q.-S. Chi,
    {\it Classification of isoparametric hypersurfaces}, preprint.

\bibitem[DX]{DX} Q. Ding, Y.L. Xin, {\it On Chern's problem for rigidity of minimal hypersurfaces in the spheres},
    Adv. Math. 227 (2011) 131--145.


\bibitem[FKM]{FKM}D. Ferus, H. Karcher and H.F. M\"{u}nzner,
    {\it Cliffordalgebren und neue isoparametrische hyperflachen}, Math. Z. 177 (1981), 479--502.

\bibitem[GB]{GB} L.C. Grove and C.T. Benson,
    {\it Finite reflection groups}, second edition, GTM 99, Springer-Verlag, 1985.

\bibitem[HH]{HH} R. Haslhofer, and O. Hershkovits,
    {\it  Ancient solutions of the mean curvature flow}. Commun. Anal. Geom. 24, 3 (2016), 593--604.

\bibitem[H]{H} G. Huisken, {\it Deforming hypersurfaces of the sphere by their mean curvature},
    Math. Z. 195 (1987), 205--219.

\bibitem[HS]{HS} G. Huisken, C. Sinestrari, {\it Convex ancient solutions of the mean curvature flow},
    J. Differential Geom., 101 (2015), 267--287.



\bibitem[LXZ]{LXZ} L. Lei, H. Xu, and E. Zhao,
    {\it Ancient solution of mean curvature flow in space forms}, arXiv:1910.05496.


\bibitem[LT]{LT} X. Liu, C.-L. Terng, {\it Mean curvature flows for isoparametric submanifolds},
    Duke Math. Journal, 147 (2009), no. 1, 157--179.

\bibitem[LN]{LN} S. Lynch and H. Nguyen,
    {\it Pinched ancient solutions to the higher codimension mean curvature flow}, arXiv:1709.09697.


\bibitem[Mu]{Mu} H. F. M\"unzner, {\it Isoparametric hypersurfaces in spheres},
    Math. Ann., 251(1980) 57--71,  256(1981) 215--232.

\bibitem[O]{O} T. Otsuki, {\it Minimal hypersurfaces in a Riemannian manifold of constant curvature},
 Amer. J. Math. 92 (1970), 145--173.

\bibitem[PT1]{PT1} C.K. Peng and C.L. Terng, {\it The scalar curvature of minimal hypersurfaces in spheres}, Ann. of Math. Stud., 103 (1983) 177--198.

\bibitem[PT2]{PT2} C.K. Peng and C.L. Terng, {\it Minimal hypersurfaces of sphere with costant scalar curvature}, Math. Ann. 266 (1983) 105--113

\bibitem[PT]{PT} U. Pinkall and G. Thorbergsson, {\it Deformations of Dupin hypersurfaces},
    Proc. Amer. Math. Soc., 107(1989), 1037--1043.

\bibitem[QT]{QT} C. Qian, and Z. Tang, {\it Clifford algebra, isoparametric foliations,
    and related geometric constructions}, arXiv:1812.10367.

\bibitem[RS]{RS} S. Risa and C. Sinestrari,
    {\it Ancient solutions of geometric flows with curvature pinching}, J. Geom. Anal. 29 (2019), 1206-1232.

\bibitem[Si]{Si}   J. Simons, {\it Minimal varieties in Riemannian manifolds}, Ann. of Math. 88 (1968) 62--105.

\bibitem[SS]{SS} D. Stryker, and A. Sun, {\it Construction of high codimension ancient mean curvature flows},
    arXiv:1908.02688.

\bibitem[T]{T} C.-L. Terng, {\it Isoparametric submanifolds and their Coxeter groups},
    J. Differential Geometry, 21 (1985), 79--107.

\bibitem[W1]{WQM} Q.-M. Wang,
      {\it On the topology of Clifford isoparametric hypersurfaces},
    J. Differential Geometry  27, (1988), 55--66.

\bibitem[W]{W} X.-J. Wang,  {\it Convex solutions to the mean curvature flow},
    Ann. of Math. (2) 173, 3 (2011), 1185--1239.

\bibitem[Wh]{Wh} B. White, {\it The nature of singularities in mean curvature flow of mean-convex
    sets}, J. Amer. Math. Soc. 16, 1 (2003), 123--138 (electronic).

\bibitem[XX]{XX}  H. Xu, Z. Xu, {\it On Chern's conjecture for minimal hypersurfaces and rigidity of self-shrinkers}, J. Functional Analysis, 273 (2017), 3406-3425

\end{thebibliography}
